\documentclass[10pt]{article}
\usepackage{mathrsfs}
\usepackage{amsfonts}
\usepackage{indentfirst,latexsym,bm}
\usepackage{amsmath, amssymb, amsthm}
\usepackage[dvips]{graphics,color}
\usepackage{graphicx}
\newtheorem{theorem}{Theorem}[section]
\newtheorem{lemma}[theorem]{Lemma}

\newtheorem{corollary}[theorem]{Corollary}
\newtheorem{proposition}[theorem]{Proposition}

\begin{document}

\title{ \bf On small bases for which $1$ has countably many expansions }
\author{\small Yuru Zou,  Lijin Wang, Jian Lu, Simon Baker}
 \date{}
\maketitle

\begin{abstract}
Let $q\in(1,2)$. A $q$-expansion
of a number $x$ in $[0,\frac{1}{q-1}]$ is a sequence $(\delta_i)_{i=1}^\infty\in\{0,1\}^{\mathbb{N}}$ satisfying

$$
x=\sum_{i=1}^\infty\frac{\delta_i}{q^i}.$$  Let $\mathcal{B}_{\aleph_0}$ denote
the set of $q$ for which there exists $x$ with a countable number
of $q$-expansions, and let $\mathcal{B}_{1, \aleph_0}$ denote the set
of $q$ for which $1$ has a countable number of  $q$-expansions. In
\cite{Sidorov6} it was shown that
 $\min\mathcal{B}_{\aleph_0}=\min\mathcal{B}_{1,\aleph_0}=\frac{1+\sqrt{5}}{2},$
and in \cite{Baker} it was shown  that
$\mathcal{B}_{\aleph_0}\cap(\frac{1+\sqrt{5}}{2}, q_1]=\{ q_1\}$,
where $q_1(\approx1.64541)$ is  the positive root of
$x^6-x^4-x^3-2x^2-x-1=0$. In this paper
 we show that the second smallest point of $\mathcal{B}_{1,\aleph_0}$ is $q_3(\approx1.68042)$, the positive root of
 $x^5-x^4-x^3-x+1=0$. Enroute to proving this result we show that $\mathcal{B}_{\aleph_0}\cap(q_1, q_3]=\{
q_2, q_3\}$, where $q_2(\approx1.65462)$ is the positive root of
$x^6-2x^4-x^3-1=0$.


\end{abstract}

\section{Introduction}
\label{intro} Let $q\in(1,2)$ and $I_{q}:=[0,\frac{1}{q-1}]$. For each $x\in I_{q}$ there exists a sequence $(\delta_i)_{i=1}^\infty\in\{0,1\}^{\mathbb{N}}$ such that
$$
x=\sum_{i=1}^\infty\frac{\delta_i}{q^i}.$$ The sequence $(\delta_i)_{i=1}^\infty$ is called a $q$-expansion for $x$. Without confusion, we
  simplify $(\delta_i)_{i=1}^\infty$ as $(\delta_i)$. It is straightforward to show that a real number $x$ has a $q$-expansion if and only if $x\in I_{q}.$

We now introduce some notation.  The so-called {\it coding map} is defined to be
$\Pi:\{0,1\}^\mathbb{N}\rightarrow I_q$ where
\begin{equation}\label{code map}
\Pi((\delta_i) )=\sum_{i=1}^\infty\frac{\delta_i}{q^i}.
\end{equation}
Throughout we let
$(\varepsilon_1...\varepsilon_n)^k$ denote the $k$ fold
concatenation of $(\varepsilon_1...\varepsilon_n)\in\{0,1\}^n$, and similarly let $(\varepsilon_1...\varepsilon_n)^\infty$ denote the
infinite concatenations of $(\varepsilon_1...\varepsilon_n).$ Given $x\in I_q$, let $\Sigma_q(x)$ denote the set of all
$q$-expansions of $x$, that is
$$
\Sigma_q(x)=\left\{(\delta_i)\in\{0,1\}^\mathbb{N}:
\Pi((\delta_i))=x\right\}.
$$
The  cardinality of the set $\Sigma_q(x)$ plays an important role in
the investigation of representations of real numbers in non-integer
bases. It was shown in
\cite{EJK} that if $q\in (1, \frac{1+\sqrt{5}}{2})$ then for each $x\in (0,\frac{1}{q-1})$ there are $2^{\aleph_0}$ different
$q$-expansions. Sidorov showed in \cite{Sidorov1, Sidorov3} that if
$q\in (1,2)$ then Lebesgue almost every $x\in (0,\frac{1}{q-1})$ has $2^{\aleph_0}$ different
$q$-expansions. Points belonging to $I_q$  with a unique
$q$-expansion were investigated in \cite{DK2, GS} for
$q\in(\frac{1+\sqrt{5}}{2},2)$. Some results concerning $x\in I_q$
having a fixed  number of $q$-expansions were established in
\cite{Baker, Sidorov5, DK1, VK, EHJ, Sidorov4, ZLL}.

Let $m\in \mathbb{N}\cup \{\aleph_0\}$ and define
$$
\mathcal{B}_m:=\left\{q\in\left[\frac{1+\sqrt{5}}{2},2\right):\text{
there exits} \;x\in I_q\;\text{ satisfying}\;
\#\Sigma_q(x)=m\right\}.
$$
Here and hereafter $\#A$ denotes the cardinality of a set $A$. The following results are known to hold:
\begin{itemize}
  \item $\min \mathcal{B}_{\aleph_0}=\frac{1+\sqrt{5}}{2}$ \cite{EHJ}.
  \item $\min \mathcal{B}_{2}=\check{q}\approx1.71064$ (the positive root of $x^4-2x^2-x-1=0$) \cite{Sidorov4}.
  \item $\min
\mathcal{B}_{k}=q_f\approx1.75488, k\geq 3$ (the positive root of
$x^3-2x^2-1=0$) \cite{Sidorov5}.
\item $\mathcal{B}_{2}\cap(\check{q},q_{f}]=\{q_{f}\}$ \cite{Sidorov5}.
\item The smallest element of $\mathcal{B}_{\aleph_0}$ strictly greater
than $\frac{1+\sqrt{5}}{2}$  is  $q_1(\approx1.64541)$ (the positive
root of $x^6-x^4-x^3-2x^2-x-1=0$) \cite{Baker}.
\end{itemize}

Understanding the $q$-expansions of $1$ is a classical problem, see \cite{GS, KL, Parry} and the references therein. The motivation of this paper is to provide a clearer understanding of what values $\#\Sigma_{q}(x)$ can take. Let
$$
\mathcal{B}_{1,m}:=\left\{q\in\left[\frac{1+\sqrt{5}}{2},2\right):
\#\Sigma_q(1)=m\right\}.
$$
It was shown in \cite{KL} that
$\min\mathcal{B}_{1,1}\approx1.78723$ (the Komornik-Loreti constant). In \cite{EHJ} it was proved that $\min\mathcal{B}_{1,\aleph_0}=\frac{1+\sqrt{5}}{2}$. For any $n\in (\mathbb{N}\setminus \{1\})\cup
\{\aleph_0\}$, Erd\"{o}s and Jo\'{o} \cite{EJ2, EJ1} constructed a
continuum of real numbers $q\in[q_0,2)$ ($q_0>1.99803$) for which
the number $1$ has precisely $n$ $q$-expansions.

Motivated by the results listed above, a natural question arises:
 what is the second smallest point of $\mathcal{B}_{1,\aleph_0}$? In this paper we will answer this question.

 Throughout this paper we let $q_1, q_2, q_3$ be as follows:  $q_1
 \approx1.64541$, $q_2 \approx1.65462$ and $q_3 \approx 1.68042$,
 which are the
positive roots of $x^6-x^4-x^3-2x^2-x-1=0$, $x^6-2x^4-x^3-1=0$ and
$x^5-x^4-x^3-x+1=0$, respectively. Our main result is the following.
\begin{theorem}\label{main theorem 1}
The smallest element of $\mathcal{B}_{1,\aleph_0}$ strictly greater
than $\frac{1+\sqrt{5}}{2}$ is $q_3$.
\end{theorem}
Enroute to proving this result we show the following.
\begin{theorem}\label{main theorem2}
 $\mathcal{B}_{\aleph_0}\cap(\frac{1+\sqrt{5}}{2}, q_3]=\{ q_1,
q_2, q_3\}$.
\end{theorem}
We remark that $\mathcal{B}_{\aleph_0}\cap(\frac{1+\sqrt{5}}{2},
q_1]=\{ q_1\}$ is known \cite{Baker}.
 The following corollary is an immediate consequence of Theorem
\ref{main theorem 1}, $\min\mathcal{B}_2\approx 1.71064$
\cite{Sidorov6}, $\min \mathcal{B}_{k}\approx1.75488$ for  $k\geq 3$
\cite{Sidorov5}, $\mathcal{B}_{2}\cap(\check{q},q_{f}]=\{q_{f}\}$
\cite{Sidorov5}, and $\min\mathcal{B}_{1,1}\approx 1.78723$
\cite{KL}.
\begin{corollary}
If $q\in(1, q_3)\setminus\{\frac{1+\sqrt{5}}{2}\}$. Then $1$ has
$2^{\aleph_0}$ different $q$-expansions.
\end{corollary}

This paper is arranged as follows. Some definitions and results from
\cite{Baker} will be recalled in Section 2. Some results from this paper will
be extended to our setting. In Section 3 we prove
Theorem \ref{main theorem2}. The final section is  devoted to the
proof of Theorem \ref{main theorem 1}.

\section{Preliminaries}
In this section we shall recall some definitions and results from
\cite{Baker}. An interpretation of $q$-expansions from the perspective of dynamical
systems was given in \cite{Baker}, see also \cite{DajaniV,DajaniV2,DajaniK1}.  Let $T_{q,0}(x)=q x$ if $x\in [0,\frac{1}{q^2-q}],$ and
let $T_{q,1}(x)=q x-1$ if  $x\in[\frac{1}{q},\frac{1}{q-1}]$. We see
that if $x\in L_q:=[0,\frac{1}{q})$ or $x\in
R_q:=(\frac{1}{q^2-q},\frac{1}{q-1}]$ then only one $T_{q,i}$ can be applied. However when $x\in S_q:=[\frac{1}{q}, \frac{1}{q^2-q}]$,
which is usually referred to as the switch region, we have a choice
between $T_{q,0}$ and $T_{q,1}$. An element of
$\bigcup_{n=0}^\infty\{T_{q,0}, T_{q,1}\}^n$ is denoted by
$\mathbf{a}$, here $\{T_{q,0}, T_{q,1}\}^0$ denotes the identity
map. Moreover, if $\mathbf{a}=a_1...a_n$ we shall use
$\mathbf{a}(x)$ to denote $a_n(\cdots( a_1(x))\cdots)$. Given $x\in I_{q}$ we
call a finite sequence of transformations $\mathbf{a}=a_1...a_n$
\emph{minimal} for $x$ if $\mathbf{a}(x)\in S_q$ and
$\mathbf{a}|i(x)\notin S_q$ for all $i<n$. Here
$\mathbf{a}|i=a_1...a_i$.  We call $\mathbf{a}(x)$ a
\emph{branching point} of $x$ if $\mathbf{a}(x)\in S_q$.

Define
$$\Omega_q(x):=\{(a_i)_{i=1}^\infty\in\{T_{q,0}, T_{q,1}\}^\mathbb{N}: a_n(\cdots( a_1(x))\cdots)\in
I_q \;\;\text{for all}\;\; n\in\mathbb{N}\}.$$ The set $\Omega_{q}(x)$ is significant because $\# \Sigma_q(x)=\#
\Omega_q(x),$ where our bijection is given by mapping $(\delta_{i})$ to $(T_{q,\delta_{i}}),$ see \cite{Baker2}.\\
\textbf{Construction of the  branching tree} The branching tree was
constructed in \cite{Baker} to study $\mathcal{B}_{\aleph_{0}}$. We
now provide details of its construction. Suppose $x\in I_q$ and
$\Omega_q(x)$ (or $\Sigma_q(x) $ ) is infinite. There exists a
unique minimal $\mathbf{a}\in \bigcup_{n=0}^\infty\{T_{q,0},
T_{q,1}\}^n$ such that $\mathbf{a}(x)\in S_q.$ Then there are two
possibilities.
\\
\textbf{Case $1$:} There exists a unique $i\in\{0,1\}$ such that $
\Omega_q(T_{q,i}(\mathbf{a}(x)))$ is finite and $
\Omega_q(T_{q,1-i}(\mathbf{a}(x)))$ is infinite. In this case, we
draw a horizontal line of finite length that then bifurcates with an
upper and lower branch.  The lower branch corresponds to
$T_{q,i}(\mathbf{a}(x))$ and stops bifurcating, the upper branch
corresponds to $T_{q,1-i}(\mathbf{a}(x))$ and goes on bifurcating.\\
\textbf{Case $2$:} Both $\Omega_q(T_{q,0}(\mathbf{a}(x)))$ and $
\Omega_q(T_{q,1}(\mathbf{a}(x)))$ are infinite. In this case, we
draw a horizontal line of finite length that then bifurcates with an
upper and lower branch.  The lower branch corresponds to
$T_{q,0}(\mathbf{a}(x))$, the upper branch corresponds to
$T_{q,1}(\mathbf{a}(x))$. Both of them go on bifurcating.\\  If
$\Omega_q(T_{q,i}(\mathbf{a}(x)))$ is infinite, as in Case $1$ or
Case $2$, then there exists a unique minimal $\mathbf{a}{'}\in
\bigcup_{n=0}^\infty\{T_{q,0}, T_{q,1}\}^n$ such that point
$\mathbf{a}{'}(T_{q,i}(\mathbf{a}(x)))$ goes to Case $1$ or Case $2$
again. This procedure continues indefinitely. The infinite tree we
construct by repeating this process is known as the infinite
branching tree corresponding to $x$. Fig.1 illustrates the
bifurcating procedure.

\begin{figure}[h]
\begin{center}
\scalebox{0.45}{\includegraphics{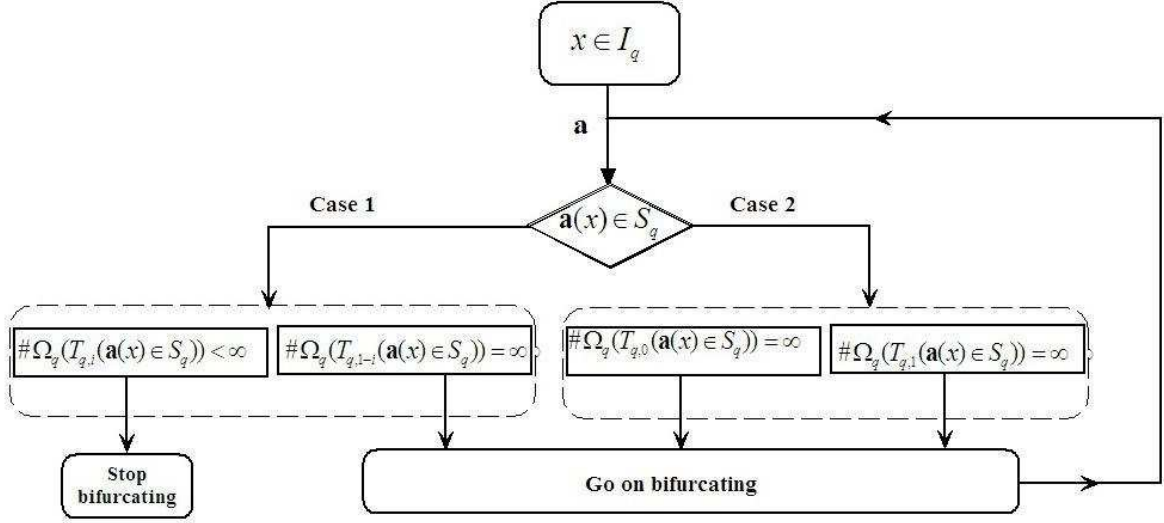}}\\
\footnotesize{\caption{The flow chart of the bifurcating procedure
corresponding to $x\in I_q$  }}\label{fig1}
\end{center}
\end{figure}

A point $x$ for which $\Sigma_{q}(x)$ is infinite is said to be a $q$
\emph{null infinite point} if for each branching point of $x$,
$\mathbf{a}(x)$ never goes to Case $2$. It is easy to check that if
$x$ is a $q$ \emph{null infinite point} then $\#
\Sigma_q(x)=\aleph_0$. The $q$ null infinite points have a critical role in the proofs of the
main results of \cite{Baker}.

\noindent\textbf{Lemma A} [1,  Proposition 2.7] \emph{ $q\in
\mathcal{B}_{\aleph_0}$ if and only if $I_q$ contains a $q$ null
infinite point}

Given $\min\mathcal{B}_2=\check{q}\approx 1.71064,$  $\min
\mathcal{B}_{k}=q_f\approx1.75488$ for $k\geq 3$, and $\mathcal{B}_{2}\cap(\check{q},q_{f}]=\{q_{f}\}$ we have that if
$q\in(\frac{1+\sqrt{5}}{2},q_f)\setminus\{\check{q}\}$ and $x$ is a
$q$ null infinite point, then for each branching point of $x$,
$\mathbf{a}(x)$, we have
$\#\Omega_{q}(T_{q,i}(\mathbf{a}(x)))=\aleph_0$ and
 $\#\Omega_{q}(T_{q,1-i}(\mathbf{a}(x)))=1$. Which implies the following inclusion
\begin{equation}\label{S_q form }\left\{\mathbf{a}(x):\mathbf{a}(x)\in S_q, \mathbf{a}\in\bigcup_{n=0}^\infty\{T_{q,0}, T_{q,1}\}^n\right\}
\subseteq \bigcup_{i=0}^1T^{-1}_{q,i}(U_q)\cap S_q.\end{equation}
Here $U_q$ denotes the set of $x\in I_q$ having a unique
$q$-expansion.

Unfortunately, it is difficult to deal with the set $S_q$. However,
by some deductions we can restrict ourselves
to a smaller set
$J_q:=[\frac{q+q^2}{q^4-1},\frac{1+q^3}{q^4-1}]\subseteq S_q$.
$J_q=S_q$ if and only if $q=q_f$.

 \noindent\textbf{Lemma B} [1,
Lemma 3.1] \emph{Let $q\in[\frac{1+\sqrt{5}}{2},q_f)$. Suppose $x\in
I_q$ satisfies $\#\Sigma_q(x)>1$, then there exists a finite
sequence of transformations $\mathbf{a}$ such that $\mathbf{a}(x)\in
J_q$. Here  $q_f(\approx 1.75488$) is the positive root of
$x^3-2x^2-1=0$.}

Proposition \ref{proposition K} is devoted to characterizing the set $\bigcup_{i=0}^1T^{-1}_{q,i}(U_q)\cap J_q $ when
$q\in[q_1, q_3]$. To prove this proposition we need Lemma C and
Lemma \ref{ykzk_lemma}.

\noindent\textbf{Lemma C} [14, Theorem 2] \emph{Let
$q\in(\frac{1+\sqrt{5}}{2}, q_f)$. Then
$$U_q=\left\{\Pi(0^k(10)^\infty), \;\Pi(1^{k}(10)^\infty),\; 0,\; \frac{1}{q-1}\right\}.$$
Where $k\geq0$.}

 Set \begin{equation*}\mathbf{y}_j=\Pi(01^{j}(10)^\infty)
 \;\;\text{and}\;\;
\mathbf{z}_j=\Pi(10^{j}(01)^\infty) \end{equation*} for $j\geq 1$. Here and
hereafter we let $(S_q\setminus
J_q)_L=[\frac{1}{q},\frac{q+q^2}{q^4-1})$ and $(S_q\setminus
J_q)_R=(\frac{1+q^3}{q^4-1}, \frac{1}{q^2-q}]$.

\begin{lemma}\label{ykzk_lemma}
Let $q\in[q_1, q_3]$. Then the following hold:\\
($i$) $\mathbf{y}_j\in J_q$ if and only if $\mathbf{z}_j\in J_q$,
and $\mathbf{y}_j\in(S_q\setminus J_q)_R $ if and only if
$\mathbf{z}_j\in(S_q\setminus J_q)_L$.\\
  $(ii)$ $T_{q,0}^m(T_{q,1}(\mathbf{y}_j))\in J_q$ if and
only if $ T_{q,1}^m(T_{q,0}(\mathbf{z}_j))\in J_q$, and
$T_{q,0}^m(T_{q,1}(\mathbf{y}_j))\in(S_q\setminus J_q)_R$ if and
only if $T_{q,1}^m(T_{q,0}(\mathbf{z}_j))\in(S_q\setminus J_q)_L$;\\
  ($iii$) $T_{q,0}^m(T_{q,1}(\mathbf{y}_j))=\mathbf{y}_k$ if and
only if $T_{q,1}^m(T_{q,0}(\mathbf{z}_j))=\mathbf{z}_k$, and
$T_{q,0}^m(T_{q,1}(\mathbf{y}_j))=\mathbf{z}_k$ if and only if
$T_{q,1}^m(T_{q,0}(\mathbf{z}_j))=\mathbf{y}_k$. Here $j\geq 1$ and
$k\geq 1$.
\end{lemma}
\begin{proof}
Direct computation shows the following equations hold.
\begin{equation}\label{ykzk_lemma euq1}\mathbf{y}_j-\frac{1+q^3}{q^4-1}=-\mathbf{z}_j+\frac{q+q^2}{q^4-1},\;\;\mathbf{y}_j-\frac{q+q^2}{q^4-1}=\frac{1+q^3}{q^4-1}-\mathbf{z}_j,
\;\;\frac{1}{q^2-q}-\mathbf{y}_j=\mathbf{z}_j-\frac{1}{q},\end{equation}
\begin{equation}\label{ykzk_lemma euq2}
T_{q,0}^m(T_{q,1}(\mathbf{y}_j))-\frac{1+q^3}{q^4-1}=-T_{q,1}^m(T_{q,0}(\mathbf{z}_j))+\frac{q+q^2}{q^4-1},
\end{equation}
\begin{equation}\label{ykzk_lemma euq3}T_{q,0}^m(T_{q,1}(\mathbf{y}_j))-\frac{q+q^2}{q^4-1}=\frac{1+q^3}{q^4-1}-T_{q,1}^m(T_{q,0}(\mathbf{z}_j)),\;\; \frac{1}{q^2-q}-T_{q,0}^m(T_{q,1}(\mathbf{y}_j))=T_{q,1}^m(T_{q,0}(\mathbf{z}_j))-\frac{1}{q}.\end{equation}

$(i)$ is implied by equation (\ref{ykzk_lemma euq1}), and $(ii)$ is
implied by equations (\ref{ykzk_lemma euq2}) and (\ref{ykzk_lemma
euq3}).

Simplifying $T_{q,0}^m(T_{q,1}(\mathbf{y}_j))=\mathbf{y}_k$ and
$T_{q,1}^m(T_{q,0}(\mathbf{z}_j))=\mathbf{z}_k,$ we see that they are both equivalent to
$$(1+2q^{-1}-q-q^{-1-j}-q^{-2-m}-q^{-1-m}+q^{-2-k-m})(q^2-1)^{-1}=0.$$
Similarly, simplifying $T_{q,0}^m(T_{q,1}(\mathbf{y}_j))=\mathbf{z}_k$ and
$T_{q,1}^m(T_{q,0}(\mathbf{z}_j))=\mathbf{y}_k$, we see that they are both equivalent to
$$(-1-2q^{-1}+q+q^{-1-j}-q^{-2-m}+q^{-2-k-m}+q^{-m})(q^2-1)^{-1}=0.$$
Thus we obtain $(iii)$.
\end{proof}

\begin{proposition}\label{proposition K}
Let $q\in[q_1, q_3]$. Then we have
\begin{eqnarray*}
\begin{array}{lll}
&\bigcup\limits_{i=0}^1T_{q,i}^{-1}(U_q)\cap J_q =\{\mathbf{y}_j,
\mathbf{z}_j, 1\leq j\leq 3 \}=\left\{\Pi(01^j(10)^\infty),
\Pi(10^j(01)^\infty), 1\leq j\leq
3\right\},\\
&\bigcup\limits_{i=0}^1T_{q,i}^{-1}(U_q)\cap (S_q\setminus J_q)_R
=\{\mathbf{y}_j, j\geq 4, \;1/(q^2-q)\}=\left\{\Pi(01^j(10)^\infty), j\geq 4,\;1/(q^2-q)\right\},\\
&\bigcup\limits_{i=0}^1T_{q,i}^{-1}(U_q)\cap (S_q\setminus
J_q)_L=\{\mathbf{z}_j, j\geq 4, \;1/q\}
=\left\{\Pi(10^j(01)^\infty), j\geq 4, \;1/q\right\}.
\end{array}
 \end{eqnarray*}
\end{proposition}
\begin{proof}
It follows from Lemma C that
$$\bigcup\limits_{i=0}^1T_{q,i}^{-1}(U_q)=\left\{\Pi(0^{j}(10)^\infty), \Pi(01^{j-1}(10)^\infty), \Pi(10^{j-1}(10)^\infty), \Pi(1^{j}(10)^\infty),
0, \frac{1}{q^2-q}, \frac{1}{q}, \frac{1}{q-1}\right\}.$$  Here
$j\geq 1$.  Then by some straightforward computation we have
\begin{equation}\label{proposition K equ3}\bigcup\limits_{i=0}^1T_{q,i}^{-1}(U_q)\cap
S_q=\left\{\mathbf{y}_j,\;\mathbf{z}_j,\; j\geq 1,
\frac{1}{q^2-q},\; \frac{1}{q} \right\}.\end{equation} Next, we
prove that
\begin{eqnarray}\label{proposition K equ4}
\mathbf{y}_j-\frac{1+q^3}{q^4-1}\left\{\begin{array}{ll}
&\leq 0\;\; \text{when} \;\;j<4\\
&>0\;\; \text{when}\;\; j\geq 4
\end{array}\right.
 \end{eqnarray}
when $q\in[q_1,q_3]$.
 It suffices to show that
 \begin{eqnarray}\label{proposition K equ5}
q^{-1}+q+q^2-q^3-q^{-1-j}-q^{1-j}\left\{\begin{array}{ll}
&\leq 0\;\; \text{when} \;\;j<4\\
&>0\;\; \text{when}\;\; j\geq 4.
\end{array}\right.
 \end{eqnarray}
By noting that
$$\mathbf{y}_j-\frac{1+q^3}{q^4-1}=(q^{-1}+q+q^2-q^3-q^{-1-j}-q^{1-j})(q^4-1)^{-1},$$
we see that (\ref{proposition K equ5}) is equivalent to
 \begin{eqnarray}\label{proposition K equ6}
\ln\frac{(q^2+1)}{-q^4+q^3+q^2+1}(\ln
q)^{-1}\left\{\begin{array}{ll}
&\geq j\;\; \text{when} \;\;j<4\\
&<j\;\; \text{when}\;\; j\geq 4.
\end{array}\right.
 \end{eqnarray}
The inequalities in (\ref{proposition K equ6}) are true when
$q\in[q_1,q_3]$, since $\ln\frac{(q^2+1)}{-q^4+q^3+q^2+1}(\ln
q)^{-1}$ is strictly increasing on the interval $[q_1, q_4]$, and
$\ln\frac{(q_4^2+1)}{-q_{4}^4+q_4^3+q_4^2+1}(\ln{q_4})^{-1}=4$ and
$\ln\frac{(q_1^2+1)}{-q_1^4+q_1^3+q_1^2+1}(\ln{q_1})^{-1}=3$. Here
 $q_4(\thickapprox 1.69784)$ is the positive root of
$x^7-x^5-x^4-2x^3-2x^2-x-1=0$. A direct computation shows that
$\mathbf{y}_j\in[\frac{q+q^2}{q^4-1},\frac{1}{q^2-q}]$ for all
$j\geq 1$. This statement combined with $(\ref{proposition K equ4})$
shows that $\mathbf{y}_j\in J_q$ if $1\leq j\leq 3$ and
$\mathbf{y}_j\in(S_q\setminus J_q)_R$ if $j\geq 4$. Using $(i)$ of
Lemma \ref{ykzk_lemma}, we also have $\mathbf{z}_j\in J_q$ if $1\leq
j\leq 3$ and $\mathbf{z}_j\in(S_q\setminus J_q)_L$ if $j\geq 4$. Our
proof now follows from (\ref{proposition K equ3}).
\end{proof}

\section{Proof of Theorem \ref{main theorem2}}
In this section we shall give a algorithm to find all elements of the set $\mathcal{B}_{\aleph_0}\cap[q_1, q_3]$. Recall that $\min\mathcal{B}_2=\check{q}\approx 1.71064$.
\begin{theorem}\label{main theorem2_2}
Let $q\in[q_1,q_f)\setminus\{\check{q}\}$ and suppose
$\bigcup\limits_{i=0}^1T_{q,i}^{-1}(U_q)\cap J_q$ is  a finite set.
Then $q\in\mathcal{B}_{\aleph_0}$ if and only if there exists $w\in
\bigcup\limits_{i=0}^1T_{q,i}^{-1}(U_q)\cap J_q$ satisfying the
following two properties. \\
$(i)$There exist a finite sequence of
transformations $\mathbf{b}\in \bigcup_{n=0}^\infty\{T_{q,0},
T_{q,1}\}^n$ such that
$$
\mathbf{b}(w)=w. $$
  $(ii)$ Let $\mathbf{b}$ be as above. Define $$B(w):=\{\mathbf{b}|i: \mathbf{b}|i(w)\in S_q, 1\leq i\leq
|\mathbf{b}|\}.$$ Then $B(w)\subset
\bigcup\limits_{i=0}^1T_{q,i}^{-1}(U_q)\cap S_q.$ Here
$|\mathbf{b}|$ denotes the length of $\mathbf{b}$.
\end{theorem}

\begin{proof}
We begin with the rightwards implication. Suppose $q\in \mathcal{B}_{\aleph_0}$. By Lemma A,
there exists $x\in(0,\frac{1}{q-1})$ such that $x$ is a $q$ null
infinite point. Furthermore, we may assume $x\in
\bigcup_{i=0}^1T_{q,i}^{-1}(U_q)\cap J_q$ by Lemma B and equation
(\ref{S_q form }). Repeatedly applying Lemma B and equation
(\ref{S_q form }), there exist
$\mathbf{a}^1,...,\mathbf{a}^m\in\bigcup_{n=0}^\infty\{T_{q,0}, T_{q
1}\}^n$ satisfying
$$ \mathbf{a}^i(\mathbf{a}^{i-1}(...(\mathbf{a}^1(x))...))\in
\bigcup\limits_{i=0}^1T_{q,i}^{-1}(U_q)\cap J_q$$ for each $1\leq
i\leq m$. Note that the set $\bigcup_{i=0}^1T_{q,i}^{-1}(U_q)\cap
J_q$ is finite. Therefore by the pigeonhole principle for $m$ sufficiently large, there must exist $i,
j\in \mathbb{N}$ with $1\leq i<j\leq m$ such that
\begin{equation}\label{theorem 3 equ2} \mathbf{a}^j(\mathbf{a}^{j-1}(...(\mathbf{a}^1(x))...))=
\mathbf{a}^i(\mathbf{a}^{i-1}(...(\mathbf{a}^1(x))...)).\end{equation}
Set $\mathbf{b}=\mathbf{a}^{i+1}\cdots\mathbf{a}^{j}$,
$\hat{\mathbf{a}}=\mathbf{a}^1\cdots\mathbf{a}^i$ and suppose
$\hat{\mathbf{a}}(x)=w$, here $w\in
\bigcup\limits_{i=0}^1T_{q,i}^{-1}(U_q)\cap J_q$. Then equation
(\ref{theorem 3 equ2}) implies that
$$\mathbf{b}(w)=w.$$ Since $x$ is a $q$ null infinite point so are all its branching points. Thus we have
the second property.

To complete our if and only if it suffices to remark that $(i)$ and
$(ii)$ imply that $w$ is a $q$ null infinite point. So $q\in
\mathcal{B}_{\aleph_0}$.
\end{proof}

Now we search for all points belonging to $\mathcal{B}_{\aleph_0}\cap[q_1, q_3]$ by applying Theorem
\ref{main theorem2_2}. Suppose $x$ is a q null infinite point and
$x\in J_q$. First we point out that
$$\bigcup\limits_{i=0}^1T_{q,i}^{-1}(U_q)\cap
J_q=\{\mathbf{y}_j=\Pi(01^{j}(10)^\infty),
\mathbf{z}_j=\Pi(10^{j}(01)^\infty), 1\leq j\leq 3\}$$ when
$q\in[q_1, q_3]$ by Proposition \ref{proposition K}. By Theorem
\ref{main theorem2_2},  we only need to consider the behavior of elements of $\bigcup\limits_{i=0}^1T_{q,i}^{-1}(U_q)\cap J_q$ under
maps belonging to $\bigcup_{n=0}^\infty\{T_{q,0}, T_{q,1}\}^n$.
 Without loss of generality we only need to consider the points
$\mathbf{y}_j$, for  $1\leq j\leq 3$ . We establish the following
lemma.

\begin{lemma}\label{ykzk_lemma2}
(i)If $q\in[q_1, q_3]$ and $\mathbf{y}_1$ is a q null infinite
point, then
\begin{equation}\label{ykzk_equ1}T_{q,1}(\mathbf{y}_1)\in L_q,\;\;
T_{q,0}^j(T_{q,1}(\mathbf{y}_1))\in L_q,  j=1,2,\;\;
T_{q,0}^3(T_{q,1}(\mathbf{y}_1))\in
\bigcup\limits_{i=0}^1T_{q,i}^{-1}(U_q)\cap J_q\end{equation}\\
(ii) If $q\in[q_1, q')$ then $\mathbf{y}_2$ is not a q null
infinite point. If $q\in[q', q_3]$ and $\mathbf{y}_2$ is a q null
infinite point, then
\begin{equation}\label{ykzk_equ2}T_{q,1}(\mathbf{y}_2)\in L_q,\;\;
T_{q,0}(T_{q,1}(\mathbf{y}_2))\in L_q,\;\;
T_{q,0}^2(T_{q,1}(\mathbf{y}_2))\in\bigcup\limits_{i=0}^1T_{q,i}^{-1}(U_q)\cap
J_q.
\end{equation}\\
Here $q'(\thickapprox 1.66184)$ is the positive root of
 $x^5-x^3-x^2-2x-2=0$.\\
(iii)If $q=q_1$ and $\mathbf{y}_3$ is a q null infinite point, then
\begin{equation}\label{ykzk_equ3}T_{q,1}(\mathbf{y}_3)\in L_q,\;\;
T_{q,0}(T_{q,1}(\mathbf{y}_3))\in\bigcup\limits_{i=0}^1T_{q,i}^{-1}(U_q)\cap
J_q.
\end{equation}
 If $q\in(q_1, q_3]$ and
$\mathbf{y}_3$ is a q null infinite point, then $q$ must be $q_3$ or
$q''(\approx 1.67365)$,
 the positive root of $x^5-2x^4+x^3-x^2+x-1=0$. \\
\end{lemma}
\begin{proof}
 Direct computation yields $(i)$. We prove $(ii)$ now. A simple computation yields
\begin{eqnarray*}
T_{q,1}(\mathbf{y}_2)\in L_q,\;\; T_{q,0}(T_{q,1}(\mathbf{y}_2))\in
L_q,\;\; T_{q,0}^2(T_{q,1}(\mathbf{y}_2))\in\left\{\begin{array}{ll}
&\Big[\frac{q+q^2}{q^{4}-1},\frac{1+q^{3}}{q^{4}-1}\Big]\;\; \text{if}\;\; q\in [q', q_3]\\
&\Big(\frac{1+q^{3}}{q^{4}-1},\frac{1}{q^2-q}\Big]\;\; \text{if}\;\; q\in[q_1,q').
\end{array}\right.
 \end{eqnarray*}
 Here $q'(\thickapprox 1.66184)$ is the positive root of
 $x^5-x^3-x^2-2x-2=0$. Thus we obtain (\ref{ykzk_equ2}).
We now assume $\mathbf{y}_2$ is $q$ null infinite for some $q\in[q_{1},q')$ and derive a contradiction. If $\mathbf{y}_2$ is a $q$ null infinite point for some $q\in[q_1, q')$ then $T_{q,0}^2(T_{q,1}(\mathbf{y}_2))\in
\bigcup\limits_{i=0}^1T_{q,i}^{-1}(U_q)\cap (S_q\setminus J_q)_R$
and it can be shown that
$$T_{q,0}(T_{q,1}(T_{q,0}^2(T_{q,1}(\mathbf{y}_2))))\in \bigcup\limits_{i=0}^1T_{q,i}^{-1}(U_q)\cap J_q.$$
Therefore by Proposition \ref{proposition K} there exists
$s\in\bigcup\limits_{i=0}^1T_{q,i}^{-1}(U_q)\cap (S_q\setminus
J_q)_R=\{\mathbf{y}_k, k\geq 4, \;1/(q^2-q)\}$ and
$u\in\bigcup\limits_{i=0}^1T_{q,i}^{-1}(U_q)\cap J_q$ such that
\begin{equation}\label{ykzk_equ4}T_{q,0}^2(T_{q,1}(\mathbf{y}_2))=s \;\;\text{and}\;\; T_{q,0}(T_{q,1}(T_{q,0}^2(T_{q,1}(\mathbf{y}_2))))=u.
\end{equation} We will show that this is not possible. That is $\mathbf{y}_2$
is not a $q$ null infinite point when $q\in[q_1, q')$. In fact,
$$T_{q,0}^2(T_{q,1}(\mathbf{y}_2))=s$$ means that
$$q^5-q^4-2q^3+2q+1=q^{-k}.$$ That is
$$k=\frac{-\ln(q^5-q^4-2q^3+2q+1)}{\ln q}.$$
The function $-\ln(q^5-q^4-2q^3+2q+1)(\ln q)^{-1}$ is  strictly
decreasing on the interval $[q_1, q')$. Table \ref{tab1} therefore implies that the only possible value of $k$ that may occur within the interval $[q_{1},q')$ is $k=4$. It is not possible that $s=1/(q^2-q)$, since $\lim\limits_{k\rightarrow \infty}
\mathbf{y}_k=1/(q^2-q)$ and $-\ln(q^5-q^4-2q^3+2q+1)(\ln q)^{-1}$ is monotonic. Table \ref{tab1} includes the values of $q$ for which $T_{q,0}(T_{q,1}(T_{q,0}^2(T_{q,1}(\mathbf{y}_2))))\in \bigcup\limits_{i=0}^1T_{q,i}^{-1}(U_q)\cap J_q$. Inspecting Table \ref{tab1} shows that there are no values of $q$ for which both equations in (\ref{ykzk_equ4}) hold. Therefore we may conclude $(ii)$.

\begin{table}[ht]
\caption{ the values of $q$  for $\mathbf{y}_2$}
\label{tab1}       
\begin{center}
\footnotesize{
\begin{tabular}{l l c c c|c|c|}
       \hline
      & $k$ & $q$ & Polynomials\\
        \hline
           $T_{q,0}^2(T_{q,1}(\mathbf{y}_2))=\mathbf{y}_k$&4    &1.65027&$x^4-x^3-x^2-x+1$\\
         &5    &1.63923&$x^6-2x^4-2x^3+x+1$ \\
        \hline
         $T_{q,0}(T_{q,1}(T_{q,0}^2(T_{q,1}(\mathbf{y}_2))))=\mathbf{y}_k$&1    &1.65637&$x^5-2x^3-x^2-x-1$\\
         &2    &1.64308&$x^5-x^4-x^3-1$ \\
        &3   &1.63420&$x^7-2x^5-x^4-x^2-x-1$\\
        \hline
            $T_{q,0}(T_{q,1}(T_{q,0}^2(T_{q,1}(\mathbf{y}_2))))=\mathbf{z}_k$&1    &1.64114&$x^2-x-1$\\
         &2    &1.65363&$x^5-2x^3-x^2+1$ \\
        &3   &1.66065&$x^5-x^4-x^3+1$\\
         \hline
\end{tabular}}
\end{center}
\end{table}

It remains to prove $(iii)$. By direct computation, we have
\begin{eqnarray*}
T_{q,1}(\mathbf{y}_3)\in L_q,\;\;
T_{q,0}(T_{q,1}(\mathbf{y}_3))\in\left\{\begin{array}{ll}
&\Big[\frac{q+q^2}{q^{4}-1},\frac{1+q^{3}}{q^{4}-1}\Big]\;\; \text{if}\;\; q=q_1\\
&\Big[\frac{1}{q},\frac{q+q^2}{q^{4}-1}\Big)\;\; \text{if}\;\; q\in(q_1,q_3]
\end{array}\right.
 \end{eqnarray*}
 and we obtain (\ref{ykzk_equ3}). Furthermore if $\mathbf{y}_3$ is a $q$ null infinite point for some $q\in(q_{1},q_{3}]$ then it is straightforward to show that
$$T_{q,1}(T_{q,0}(T_{q,0}(T_{q,1}(\mathbf{y}_3))))\in
\bigcup\limits_{i=0}^1T_{q,i}^{-1}(U_q)\cap J_q.$$ Moreover, there must exist  $s\in
\bigcup\limits_{i=0}^1T_{q,i}^{-1}(U_q)\cap (S_q\setminus
J_q)_L=\{\mathbf{z}_k,k\geq 4,\; 1/q\}$ and  $u\in
\bigcup\limits_{i=0}^1T_{q,i}^{-1}(U_q)\bigcap J_q$  such that
\begin{equation}\label{ykzk_equ5} T_{q,0}(T_{q,1}(\mathbf{y}_3))=s \;\;\text{and}\;\; T_{q,1}(T_{q,0}(T_{q,0}(T_{q,1}(\mathbf{y}_3))))=u.
\end{equation} The equation $T_{q,0}(T_{q,1}(\mathbf{y}_3))=s$ means
that
$$k=\frac{-\ln(1-q^{-1}+q^2+q^3-q^4)}{\ln q}.$$
The function $-\ln(1-q^{-1}+q^2+q^3-q^4)(\ln q)^{-1}$ is strictly
increasing on the interval $(q_1, q_3)$ and
\begin{equation}\label{ykzk_equ6}\lim_{q\rightarrow q_{3}}-\ln(1-q^{-1}+q^2+q^3-q^4)(\ln
q)^{-1}=+\infty.\end{equation} Table \ref{tab2} records the first
few solutions of $-\ln(1-q^{-1}+q^2+q^3-q^4)(\ln q)^{-1}=k$. It is
easy to show that the case where $s=1/q$ is only possible  when
$q=q_{3}$.  In Table \ref{tab2}, we also list the $q$'s for which
$T_{q,1}(T_{q,0}(T_{q,0}(T_{q,1}(\mathbf{y}_3))))=u$ holds. By
inspecting Table \ref{tab2} and using the fact
$-\ln(1-q^{-1}+q^2+q^3-q^4)(\ln q)^{-1}$ is increasing with $q,$ we
see that the only values of $q$ for which both equations in
(\ref{ykzk_equ5}) hold simultaneously are $q\approx 1.67365$ and
when $q=q_3$.

\end{proof}

\begin{table}[ht]
\caption{ the values of $q$ for $\mathbf{y}_3$}
\label{tab2}       
\begin{center}
\footnotesize{
\begin{tabular}{l l c c c|c|c|}
       \hline
        & $k$ & $q$ & Polynomials\\
        \hline
         $T_{q,0}(T_{q,1}(\mathbf{y}_3))=\mathbf{z}_k$&4    &1.66041&$x^6-x^5-x^3-x^2-1$\\
         &5    &1.66883&$x^8-x^6-x^5-2x^4-x^3-x^2-x-1$ \\
        &6   &1.67365&$x^5-2x^4+x^3-x^2+x-1$\\
        &7   &1.67644&$x^{10}-x^8-x^7-2x^6-x^5-x^4-x^3-x^2-x-1$\\
        $T_{q,0}(T_{q,1}(\mathbf{y}_3))=q^{-1}$&  &1.68042&$x^5-x^4-x^3-x+1$\\
        \hline
            $T_{q,1}(T_{q,0}(T_{q,0}(T_{q,1}(\mathbf{y}_3))))=\mathbf{y}_k$ &1    &1.68042&$x^5-x^4-x^3-x+1$\\
         &2    &1.65963&$x^7-2x^5-x^4-x^3+x+1$ \\
        &3   &1.64541&$x^6-x^4-x^3-2x^2-x-1$\\
        \hline
           $T_{q,1}(T_{q,0}(T_{q,0}(T_{q,1}(\mathbf{y}_3))))=\mathbf{z}_k$ &1    &1.65462&$x^6-2x^4-x^3-1$\\
         &2    &1.67365&$x^5-2x^4+x^3-x^2+x-1$ \\
        &3   &1.68400&$x^8-2x^6-x^5-x^2-x-1$\\
        \hline

\end{tabular}}
  \end{center}
\end{table}

Table \ref{tab3} lists the values of $q$ for which equations (\ref{ykzk_equ1}-\ref{ykzk_equ3}, \ref{ykzk_equ5}) hold true
independently. In fact, it follows from Lemma \ref{ykzk_lemma2} and
the  symmetric property of $\mathbf{y}_k$ and $\mathbf{z}_k$ shown
in Lemma \ref{ykzk_lemma} that equations in
(\ref{ykzk_equ1}-\ref{ykzk_equ3}, \ref{ykzk_equ5}) give all the
possible values of $q$ such that $q\in
\mathcal{B}_{\aleph_0}\cap[q_1,q_3]$.

\begin{table}[ht]
\caption{ the values of $q$ for $\mathbf{y}_k$ and $ \mathbf{z}_k,
1\leq k\leq 3$}
\label{tab3}       
\begin{center}
\footnotesize{
\begin{tabular}{l l c c c|c|c|}
       \hline
       & $k$ & $q$ & Polynomials\\
        \hline
        $ T_{q,0}^3(T_{q,1}(\mathbf{y}_1))=\mathbf{y}_k$ (or $T_{q,1}^3(T_{q,0}(\mathbf{z}_1))=\mathbf{z}_k$)&1    &1.68042&$x^5-x^4-x^3-x+1$ \\
         &2    &1.65963 &$x^7-2x^5-x^4-x^3+x+1$\\
         &3   &1.64541 &$x^6-x^4-x^3-2x^2-x-1$\\
        \hline
      $T_{q,0}^3(T_{q,1}(\mathbf{y}_1))=\mathbf{z}_k$ (or $T_{q,1}^3(T_{q,0}(\mathbf{z}_1))=\mathbf{y}_k$)&1    &1.65462&$x^6-2x^4-x^3-1$ \\
         &2    &1.67365 &$x^5-2x^4+x^3-x^2+x-1$\\
         &3   &1.68400 &$x^8-2x^6-x^5-x^2-x-1$\\
        \hline
           $T_{q,0}^2(T_{q,1}(\mathbf{y}_2))=\mathbf{y}_k$ (or $T_{q,1}^2(T_{q,0}(\mathbf{z}_2))=\mathbf{z}_k$)&1    &1.72208&$x^4-x^3-x^2-x+1$\\
         &2    &1.68929&$x^6-2x^4-2x^3+x+1$ \\
        &3   &1.6663&$x^6-x^5-x^4-x^3+x^2+1$\\
        \hline
         $T_{q,0}^2(T_{q,1}(\mathbf{y}_2))=\mathbf{z}_k$ (or $T_{q,1}^2(T_{q,0}(\mathbf{z}_2))=\mathbf{y}_k$)&1    &1.67602&$x^5-2x^3-x^2-1$\\
         &2    &1.7049&$x^5-x^4-x^3-1$ \\
        &3   &1.72004&$x^7-2x^5-x^4-x^2-x-1$\\
        \hline
          $T_{q,0}(T_{q,1}(\mathbf{y}_3))=\mathbf{z}_k$ (or $T_{q,1}(T_{q,0}(\mathbf{z}_3))=\mathbf{y}_k)$ &3   &1.64541&$x^6-x^4-x^3-2x^2-x-1$\\
          \hline
           $T_{q,0}(T_{q,1}(\mathbf{y}_3))=\mathbf{z}_6,\;T_{q,1}(T_{q,0}(\mathbf{z}_6))=\mathbf{z}_2$
           \\(or $T_{q,1}(T_{q,0}(\mathbf{z}_3))=\mathbf{y}_6$, $T_{q,0}(T_{q,1}(\mathbf{y}_6))=\mathbf{y}_2$)
         &    &1.67365&$x^5-2x^4+x^3-x^2+x-1$\\
          $T_{q,0}(T_{q,1}(\mathbf{y}_3))=1/q,\;T_{q,1}(T_{q,0}(1/q))=\mathbf{y}_1$
           \\(or $T_{q,1}(T_{q,0}(\mathbf{z}_3))=1/(q^2-q)$ , $T_{q,0}(T_{q,1}(1/(q^2-q)))=\mathbf{z}_1$)
         &  &1.68042&$x^5-x^4-x^3-x+1$\\
        \hline
\end{tabular}}
\end{center}
\end{table}

We  are now ready to prove Theorem \ref{main theorem2}.

\noindent\textbf{Proof of Theorem \ref{main theorem2}.} By applying
Theorem \ref{main theorem2_2} and Lemma \ref{ykzk_lemma2}, we can
find all points belonging to $\mathcal{B}_{\aleph_0}\cap[q_1,
q_3]$ in Table \ref{tab3}. One can see from Table \ref{tab3} that
$$T_{q,0}(T_{q,1}(\mathbf{y}_3))=\mathbf{z}_3, T_{q,1}(T_{q,0}(\mathbf{z}_3))=\mathbf{y}_3, T_{q,1}(T_{q,0}(T_{q,0}(T_{q,1}(\mathbf{y}_3))))=\mathbf{y}_3,$$
when $q_{1}(\approx 1.64541)$, which is the positive root of
$x^6-x^4-x^3-2x^2-x-1=0$,
$$T_{q,0}^3(T_{q,1}(\mathbf{y}_1))=\mathbf{z}_1, T_{q,1}^3(T_{q,0}(\mathbf{z}_1))=\mathbf{y}_1,  T_{q,0}^3(T_{q,1}^4(T_{q,0}(\mathbf{z}_1)))=\mathbf{z}_1,$$
when $q_{2}(\approx 1.65462)$, which is the positive root of
$x^6-2x^4-x^3-1=0$, and
$$T_{q,0}^3(T_{q,1}(\mathbf{y}_1))=\mathbf{y}_1,$$
when $q_{3}(\approx 1.68042)$, which is the positive root of
$x^5-x^4-x^3-x+1=0$. So, the conditions in Theorem \ref{main
theorem2_2} are satisfied when $q=q_1, q_2,q_3$, respectively. That
is $q_{j}\in \mathcal{B}_{\aleph_0}$ for each $1\leq j\leq 3$.

 Finally, it is easy to check that  $q\notin \mathcal{B}_{\aleph_0}$ if $q$ takes the values listed in Table \ref{tab3}, except when $q=q_j, 1\leq j\leq
 3$.
For example, we have
$$T_{q,0}^3(T_{q,1}(\mathbf{y}_1))=\mathbf{y}_2$$  when $q(\approx 1.65963)$ is the positive root of
$x^5-x^4-x^3-x+1=0$. However,
  there exists no point in $\bigcup\limits_{i=0}^1T_{q,i}^{-1}(U_q)\bigcap
J_q$ such that condition $(i)$ of Theorem \ref{main theorem2_2}
holds for this value of $q.$ The other values of $q$ are dealt with similarly.
 $\hfill\Box$

\section{Proof of Theorem \ref{main theorem 1}}
To prove Theorem \ref{main theorem 1} it suffices to prove $1$ has
$\aleph_0$
 $q$-expansions when $q=q_3$ and $1$ a continuum of $q$-expansions when $q=q_1$ and $q=q_2$.
We only prove the case of $q=q_1$, the case of $q=q_2$ can be
verified in a  similar way.

\begin{theorem} If $q_3 (\approx 1.68042)$ is the positive root of
$x^5-x^4-x^3-x+1=0$. Then $1$  is a $q_3$ null infinite point and
therefore has $\aleph_{0}$ $q_3$-expansions.
\end{theorem}
\begin{proof}
It is straightforward to show that for all $k\geq 0$
$$(T_{q_{3},0}^{3}\circ T_{q_{3},1})^{k}(T_{q_{3},1}(1))\in S_{q_{3}}.$$ Moreover these are the only $\mathbf{a}\in \bigcup_{n=0}^{\infty}\{T_{q_{3},0},T_{q_{3},1}\}^{n}$ such that $\mathbf{a}(1)\in S_{q_{3}}.$ This is sufficient to imply that $1$ is a $q_{3}$ null infinite point . Namely that for each $k\geq 0$ we have $$T_{q_{3},0}\Big((T_{q_{3},0}^{3}\circ T_{q_{3},1})^{k}(T_{q_{3},1}(1))\Big)=\Pi(1(10)^{\infty}).$$ Therefore
$$\Sigma_{q_3}(1)=\left\{1(10^3)^k01(10)^\infty, 1(10^3)^\infty
\right\}.$$
\end{proof}Figure $2$ demonstrates the construction of $\Sigma_{q_3}(1)$.

\begin{figure}[h]
\begin{center}
\scalebox{0.4}{\includegraphics{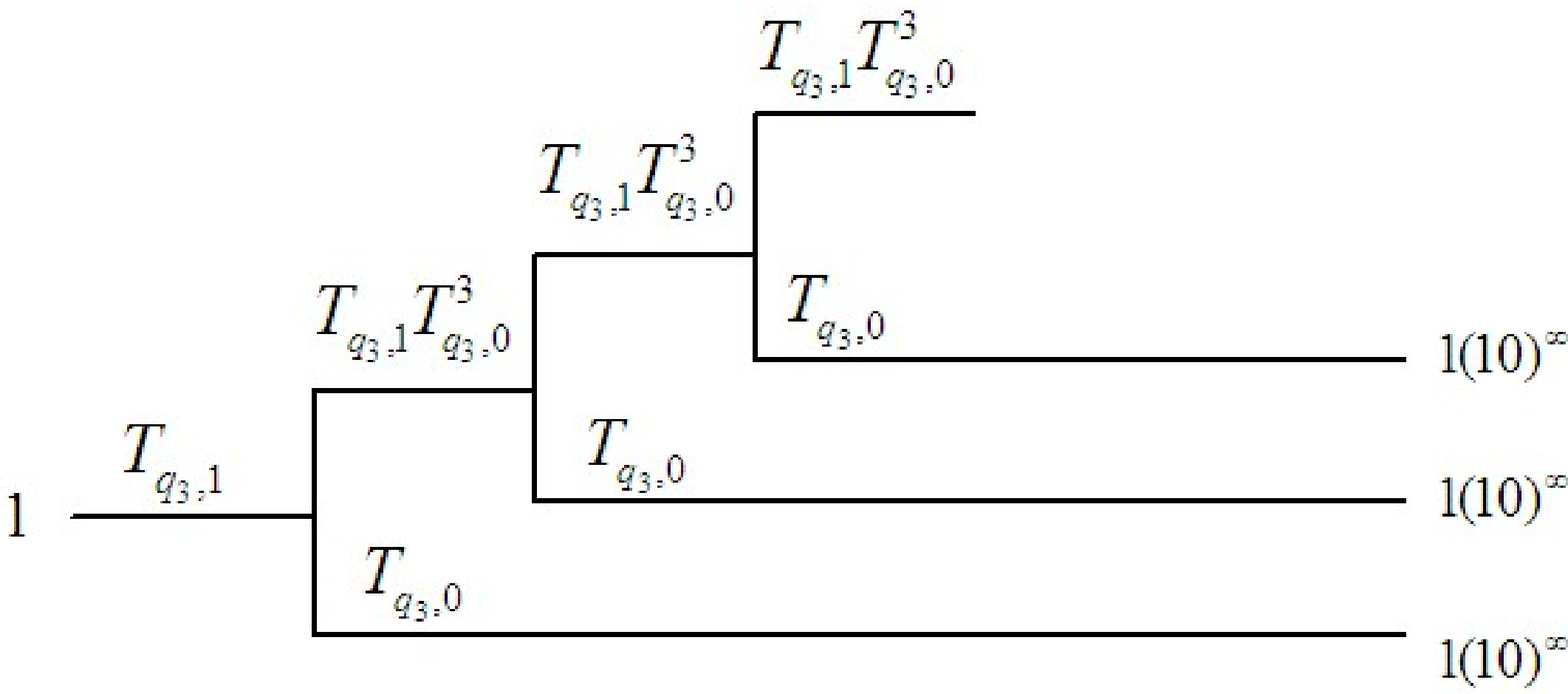}}\\
\footnotesize{\caption{The construction of $\Sigma_{q_3}(1)$
 }}\label{fig1}
\end{center}
\end{figure}

\begin{theorem}\label{main theorem 1_1}
Let $q_1 (\approx 1.64541)$ be the positive solution of the equation
$x^6-x^4-x^3-2x^2-x-1=0$. Then $1$ has $2^{\aleph_{0}}$
$q_1$-expansions.
\end{theorem}
\begin{proof}
We proceed via a proof by contradiction. We assume that $1$ has $\aleph_{0}$ $q_{1}$-expansions and obtain a contradiction. It is a simple calculation to show that $$w:=(T_{q_{1},0}^5\circ T_{q_{1},1}^{2})(1)\in S_{q_{1}}.$$ Therefore $w$ has either $\aleph_{0}$ $q_{1}$-expansions or $2^{\aleph_{0}}$ $q_{1}$-expansions. Since we have assumed $1$ has $\aleph_{0}$ $q_{1}$-expansions, $w$ must also have $\aleph_{0}$ $q_{1}$-expansions. Therefore $w$ can be mapped to a $q_{1}$ null infinite point, and by Lemma $C$ it can be mapped to a point with a periodic $q_{1}$-expansion. The above implies that $1$ has a $q_{1}$ expansion that begins $(1100000(\delta_{i})_{i=1}^{\infty})$ where $(\delta_{i})_{i=1}^{\infty}$ is eventually periodic. This is obviously equivalent to
\begin{equation}\label{1 equation}
1=\frac{1}{q_{1}}+\frac{1}{q_{1}^{2}}+\frac{1}{q_{1}^{7}}\sum_{i=1}^{\infty}\frac{\delta_{i}}{q_{1}^{i}}.
\end{equation} Since $(\delta_{i})_{i=1}^{\infty}$ is eventually periodic we may use properties of geometric series to deduce that (\ref{1 equation}) is equivalent to
\begin{equation}\label{2 equation}
1=\frac{1}{q_{1}}+\frac{1}{q_{1}^{2}}+\frac{f(q_{1})}{g(q_{1})}
\end{equation} where $f(x),g(x)\in \mathbb{Z}[x].$ Equation (\ref{2 equation}) is just an algebraic relation and so must also be satisfied by the conjugates of $q_{1}$, that is the other roots of $x^6-x^4-x^3-2x^2-x-1=0.$ We now show that this cannot be the case for a particular choice of conjugate, namely $q_{1}^*\approx -1.20458.$ Equation (\ref{2 equation}) is equivalent to equation (\ref{1 equation}), so (\ref{1 equation}) must also hold with $q_{1}$ replaced by $q_{1}^{*}.$ We observe the following
\begin{align*}
1&=\frac{1}{q_{1}^*}+\frac{1}{(q_{1}^*)^2}+\frac{1}{(q_{1}^*)^7}\sum_{i=1}^{\infty}\frac{\delta_{i}}{(q_{1}^*)^i}\\
& \leq \frac{1}{q_{1}^*}+\frac{1}{(q_{1}^*)^2}+\frac{1}{(q_{1}^*)^7}\frac{q_{1}^{*}}{(q_{1}^*)^2-1}\\
& <1.
\end{align*}
Where the final strict inequality follows from a simple calculation. Thus we have our desired contradiction.
\end{proof} The proof that $q_{2}\notin \mathcal{B}_{1,\aleph_{0}}$ is done analogously. In this case we similarly use a conjugate of $q_{2},$ namely the number $q_{2}^*\approx -1.26493.$



 YURU ZOU,  LIJIN WANG, JIAN LU: COLLEGE OF MATHEMATICS AND COMPUTATIONAL SCIENCE,
SHENZHEN UNIVERSITY, SHENZHEN 518060, CHINA

SIMON BAKER: SCHOOL OF MATHEMATICS, THE UNIVERSITY OF MANCHESTER,
OXFORD ROAD, MANCHESTER M13 9PL, UNITED KINGDOM

E-mail: yrzou@163.com, ljwang1989@126.com, jianlu1979@163.com,
simon.baker@manchester.ac.uk

\end{document}